\documentclass[11pt,a4paper]{amsart}
\usepackage{cancel}
\usepackage[normalem]{ulem}
\usepackage{amssymb}
\usepackage{latexsym}
\usepackage{exscale}
\usepackage{amsfonts}
\usepackage{graphicx}
\usepackage{mathrsfs}
\usepackage{amsmath,amscd,amsthm}
\usepackage{bbm,pifont}
\usepackage{enumerate}
\usepackage{color}
\usepackage{amssymb}
\usepackage{latexsym}
\usepackage{exscale}
\usepackage[utf8]{inputenc}
\usepackage[T1]{fontenc}
\usepackage{dsfont}
\usepackage[mathscr]{eucal}
\usepackage[dvipsnames]{xcolor}

\usepackage[colorlinks, citecolor=blue,hypertexnames=false]{hyperref}
 \makeatletter
\renewcommand\@makefnmark{%
  \hbox{\@textsuperscript{\normalfont\color{black}\@thefnmark}}}
\makeatother

\allowdisplaybreaks

\DeclareMathAlphabet\mathcalbf{OMS}{cmsy}{b}{n}
\DeclareMathAlphabet\EuScript{U}{eus}{m}{n}
\DeclareMathAlphabet\EuScriptBold{U}{eus}{b}{n}

\numberwithin{equation}{section}
\newtheorem{theorem}{Theorem}[section]
\newtheorem{lemma}[theorem]{Lemma}

\newtheorem{proposition}[theorem]{Proposition}

\newtheorem{definition}[theorem]{Definition}

\newtheorem{remark}[theorem]{Remark}

\newcommand{\Cine}{\EuScript C_\epsilon}
\newcommand{\Rine}{\EuScript R_\epsilon}

\newcommand{\Sopo}{\EuScript S_\omega}
\newcommand{\SopO}{\EuScript S_{\Ot}}
\newcommand{\Op}{\Omega_p}
\newcommand{\Ot}{\Omega_2}
\newcommand{\pp}{\psi_p}
\newcommand{\pt}{\psi_2}
\newcommand{\Tinv}{ (\mathcal T^s_\epsilon)^{-1}}

\def\C{\mathbb C}

\begin{document}
\allowdisplaybreaks

\title[Weighted Cauchy--Szeg\H{o} Projection]
{The Cauchy--Szeg\H{o} Projection  for domains\\ in $\mathbb C^n$ with minimal smoothness:\\ 
 weighted theory}

\author{Xuan Thinh Duong, Loredana Lanzani, Ji Li and Brett D. Wick}

\address{Xuan Thinh Duong, Department of Mathematics, Macquarie University, NSW, 2109, Australia}
\email{xuan.duong@mq.edu.au}

\address{Loredana Lanzani, University of Bologna, Piazza di Porta S. Donato 5, 40126 Bologna, Italy}
\email{loredana.lanzani2@unibo.it}

\address{Ji Li, Department of Mathematics, Macquarie University, NSW, 2109, Australia}
\email{ji.li@mq.edu.au}

\address{Brett D. Wick, Department of Mathematics \& Statistics\\
         Washington University - St. Louis\\
         St. Louis, MO 63130-4899 USA
         }
\email{wick@math.wustl.edu}


\begin{abstract}
Let $D\subset\C^n$ be a bounded, strongly pseudoconvex domain whose boundary $bD$ satisfies the minimal regularity condition of class $C^2$.  A 2017 result of Lanzani \& Stein \cite{LS2017} states that 
the Cauchy--Szeg\H{o} projection $\Sopo$ defined with respect to a bounded,
positive continuous multiple $\omega$ of induced Lebesgue measure,
  {maps $L^p(bD, \omega)$ to $L^p(bD, \omega)$ continuously} for any $1<p<\infty$.
  Here we show that $\Sopo$ 
  satisfies explicit quantitative bounds in $L^p(bD, \Op)$, for any $1<p<\infty$ and for any $\Op$ in the maximal
  class
   of \textit{$A_p$}-measures, that is for $\Omega_p = \psi_p\sigma$ where  $\psi_p$ is a Muckenhoupt $A_p$-weight and $\sigma$ is the induced Lebesgue measure (with 
   $\omega$'s as above being a sub-class). Earlier results 
  rely upon an asymptotic expansion and subsequent pointwise estimates of the Cauchy--Szeg\H o kernel, but these are unavailable in our setting of minimal regularity {of $bD$}; at the same time,
 more recent 
  techniques 
  that allow to
   handle domains with minimal regularity  \cite{LS2017}
   are not applicable to $A_p$-measures.
    It turns out that the method of {quantitative} extrapolation is an appropriate replacement for the missing tools.

 To finish, we identify a class of holomorphic Hardy spaces defined with respect to $A_p$-measures for which a meaningful notion of Cauchy-Szeg\H o projection can be defined when $p=2$.
\end{abstract}

\maketitle


\centerline{\em In memory of J. J. Kohn}
\vskip0.2in

\section{Introduction}
Given a domain $D\Subset \mathbb C^n$ with rectifiable boundary and a reference measure $\mu$ on $bD$ (the boundary of $D$), the {\em Cauchy-Szeg\H o projection $S_\mu$} is the (unique) orthogonal projection of the Hilbert space $L^2(bD, \mu)$ onto the holomorphic Hardy space $H^2(bD, \mu)$. The {\em $L^p(bD, \mu)$- regularity problem for $S_\mu$} that is, the study of the boundedness of $S_\mu$  on $L^p(bD, \mu)$ with $p\neq 2$, is a deep problem in harmonic analysis whose solution is highly dependent on the analytic and geometric properties of $D$ and of the reference measure $\mu$; see e.g., \cite{KS}, \cite{KL2}, \cite{LS2004}, \cite{LS2017}, \cite{PS}, \cite{WW}.

In this paper we 
present
 a new approach to the $L^p$-regularity problem for the Cauchy--Szeg\H o projection $\Sopo$ defined with respect to $\omega = \Lambda \, \sigma$ (any) bounded, positive continuous multiple of the induced Lebesgue measure $\sigma$ associated to a 
 strongly pseudoconvex domain $D\Subset \mathbb C^n$, with $n\geq 2$, that satisfies the minimal regularity condition of class $C^2$. As an application, we obtain the quantitative bound
\begin{equation}\label{E:main-1}
\|\Sopo\, g\|_{L^p(bD, \Op)}\lesssim [\Omega_p]_{A_p}^{3\cdot\max\left\{1,\, \frac{1}{p-1}\right\}}
\,\|g\|_{L^p(bD, \Op)}, \quad 1<p<\infty,\quad 
\end{equation}
for any $\omega\in \{\Lambda\sigma\}_\Lambda$ as above and for any $\Op$ in the class of Muckenhoupt measures $A_p(bD)$, where $[\Op]_{A_p}$ stands for the $A_p$-character of $\Op$
  while the implicit constant depends solely on $p$, $\omega$ and $D$; see Theorem \ref{T:Szego-for-A_p} for the precise statement.  We point out that the power $3$ can be sharpened as $2+\delta$ for any $\delta>0$. However, it can not be reduced to $2$ due to the structure of minimal smoothness of the domain. By contrast,
 the norm bounds 
  obtained in the reference result \cite{LS2017}:
\begin{equation}\label{E:LS-o}
\|\Sopo\, g\|_{L^p(bD, \omega)}
\leq C(D, \omega, p)\, 
\|g\|_{L^p(bD, \omega)}, \quad 1<p<\infty
\quad \omega\in \{\Lambda\sigma\}_\Lambda \ \text{as above,}
\end{equation}
and  its recent improvement for $\omega:= \sigma$, \cite[Theorem 1.1]{WW}:
\begin{equation}\label{E:WW}
\|\EuScript S_\sigma\, g\|_{L^p(bD, \Op)}\leq
\tilde C(D, [\Op]_{A_p})
\|g\|_{L^p(bD, \Op)},\quad 1<p<\infty,\quad \Op\in A_p(bD)
\end{equation}
 are unspecified functions of the stated variables.

As is well known, the $A_p$-measures 
are the maximal (doubling) measure-theoretic
 framework 
for a great variety of singular integral operators. Any positive, continuous multiple of $\sigma$ is, of course, a member of $A_p(bD)$ for {\em any} $1<p<\infty$, but the class $\{A_p(bD)\}_p$ far encompasses the family $\{\Lambda\,\sigma\}_\Lambda$.
 The classical methods for $\Sopo$ rely on the asymptotic expansion of the Cauchy--Szeg\H o kernel, {see e.g., the foundational works \cite{FEF}, \cite{NRSW} and \cite{PS},} however this expansion is not available  if $D$ is non-smooth (below the class $C^3$). What's more, the Cauchy--Szeg\H o projection $\Sopo$
may not be Calder\'on--Zygmund, see \cite{LS2004}, thus none of \eqref{E:main-1}--\eqref{E:WW}  can follow by a direct application of
the Calder\'on--Zygmund theory. The proof of \eqref{E:LS-o} given in \cite{LS2017}  starts with the comparison of $\Sopo$
with the members of an ad-hoc
 family $\{\Cine\}_\epsilon$ of non-orthogonal projections (the so-called Cauchy--Leray integrals):
\begin{equation}\label{E:prelim}
\Cine = \Sopo\circ\big[ I-(\Cine^\dagger - \Cine)\big]\quad \text{in}\quad L^2(bD, \omega),\quad 0<\epsilon<\epsilon (D),
\end{equation}
where the upper-script $\dagger$ denotes the adjoint in $L^2(bD, \omega)$. The operators $\{\Cine\}_\epsilon$ are bounded on $L^p(bD, \omega)$, $1<p<\infty$, by an application of the $T(1)$ theorem. Furthermore, an elementary, Hilbert space-theoretic observation yields the factorization
\begin{equation}\label{E:LS-1}
\Sopo = \Cine\circ\big[ I-(\Cine^\dagger - \Cine)\big]^{-1}\quad \text{in}\quad L^2(bD, \omega),\quad 0<\epsilon<\epsilon (D),
\end{equation}
along with its refinement
\begin{equation}\label{E:LS-2}
  \Sopo = \bigg(\EuScript C_\epsilon +
   \Sopo\circ\big((\EuScript R^s_\epsilon)^\dagger-\EuScript R^s_\epsilon\big)\bigg)\circ 
   \big[I-\big((\Cine^s)^\dagger-\Cine^s\big)\big]^{-1}\quad \text{in}\ \ L^2(bD, \omega), \quad 0<\epsilon<\epsilon (D).
\end{equation}
For the latter, the Cauchy--Leray integral $\Cine$ is decomposed as the sum
$\Cine = \Cine^s + \Rine^s$ where the principal term $\Cine^s$  (bounded, by another application of the $T(1)$ theorem) enjoys  the cancellation 
\begin{equation}\label{E:oldcancellation1}
\|(\EuScript C_\epsilon^s)^\dagger-\EuScript C_\epsilon^s\|_{L^p(bD,\,\omega)\to L^p(bD,\,\omega)} 
  \leq\, \epsilon^{1/2}\, C(p, D, \omega),\quad 1<p<\infty\, ,\quad 0<\epsilon<\epsilon (D).
  \end{equation}
By contrast, the remainder $\EuScript R_\epsilon^s$ has $L^p\to L^p$ norm that may blow up as $\epsilon\to 0$ but is weakly smoothing in the sense that
\begin{equation}\label{E:oldcancellation2}
 \EuScript R_\epsilon^s\ \text{and}\  (\EuScript R_\epsilon^s)^\dagger\ \text{are bounded}: L^1(bD, \omega)\to L^\infty (bD, \omega)\quad \text{for any}\quad 
0<\epsilon<\epsilon (D).
\end{equation}
Combining \eqref{E:oldcancellation1} and \eqref{E:oldcancellation2} with the bounded inclusions
 \begin{equation}\label{E:Lp-incl}
L^2(bD, \omega)\ \hookrightarrow\  L^p(bD, \omega)\ \hookrightarrow\ L^1(bD, \omega),\quad 1<p\leq 2,
\end{equation}
one concludes that there is $\epsilon = \epsilon (p)$ such that the right hand side of 
  \eqref{E:LS-2}, and therefore $\Sopo$, extends to a bounded operator on
$L^p(bD, \omega)$ for any $\omega\in \{\Lambda\,\sigma\}_\Lambda$ whenever $1<p< 2$;
 the proof for the full $p$-range
 then follows by duality.
\vskip0.1in
Implementing this argument to {\em all} $A_p$-measures presents conceptual difficulties: for instance, there are no analogs of \eqref{E:oldcancellation2} nor of \eqref{E:Lp-incl} that are valid with $\Omega_p$ in place of $\omega$ because $A_p$-measures
 may change with $p$.
In \cite[Lemma 3.3]{WW} it is shown that \eqref{E:oldcancellation1} is still valid for $p=2$ if one replaces the measure $\omega$ in the $L^2(bD)$-space 
with an $A_2$-measure; the conclusion \eqref{E:WW} then follows by Rubio de Francia's original
extrapolation theorem \cite{R}.
\vskip0.1in

Here 
we obtain a new cancellation in $L^p(bD, \Omega_p)$ where the dependence of the norm-bound on the $A_p$-character of the measure is completely explicit, namely
\begin{equation}\label{E:newcancellation-2}
\|(\EuScript C_\epsilon^s)^\dagger-\EuScript C_\epsilon^s\|_{L^p(bD,\, \Op)\to L^p(bD,\, \Op)} 
  \lesssim \epsilon^{1/2} [\Op]_{A_p}^{\max\{1,{1\over p-1}\}},\quad 
  1<p<\infty
  \end{equation}
where 
 the implied constant depends on $p$, $D$ and $\omega$ {but is independent of $\Op$ and of $\epsilon$};
 see Proposition \ref{prop cancellation} for the precise 
  statement. 
  Combining 
   \eqref{E:newcancellation-2} 
  with the reverse H\"older inequality for $A_2(bD)$
   we obtain
 \begin{equation}\label{E:extrap}
  \|\Sopo \|_{L^2(bD, \Ot)\to L^2(bD, \Ot)}\lesssim\, [\Ot]_{A_2}
  \quad \text{for any}\ \ \Ot\in A_2(bD)\ \ \text{and any}\ \ \omega\in \{\Lambda\,\sigma\}_\Lambda.
 \end{equation}
 The conclusion \eqref{E:main-1}
  then follows by quantitative extrapolation \cite[Theorem 9.5.3]{Gra}, 
  which serves as an appropriate replacement for \eqref{E:oldcancellation2} and \eqref{E:Lp-incl}. Incidentally,
 we are not aware of other applications of extrapolation to several complex variables; yet quantitative extrapolation seems to provide an especially well-suited approach to the analysis of
 orthogonal projections onto spaces of holomorphic functions\footnote{the Cauchy--Szeg\H o and Bergman projections being two prime examples.} because orthogonal projections are naturally bounded, with minimal norm, on the Hilbert space where they are defined; thus there is always a baseline choice
  of $\Ot$ for which \eqref{E:extrap} holds trivially 
  (that is, with $[\Ot]_{A_2} =1$).
We anticipate that this approach will give new insight into other settings where this kind of $L^p$-regularity problems are unsolved, such as the strongly $\mathbb C$-linearly convex domains of class $C^{1,1}$ (\cite{LS2014}) and the $C^{1,\alpha}$ model domains of \cite{LS2019}.

 \vskip0.1in
 
{\em Can the Cauchy--Szeg\H o projection be defined for measures other than $\omega\in \{\Lambda\,\sigma\}_\Lambda$?} Orthogonal projections are highly dependent upon the choice of reference measure for the Hilbert space where they are defined: it is therefore natural to seek
 the maximal measure-theoretic framework for which the notion of Cauchy--Szeg\H o projection is meaningful.
The 
$A_p$-measures are an obvious choice in view of their historical relevance to the theory of singular integral operators. It turns out that in this context the Cauchy--Szeg\H o projection is meaningful only if it is defined with respect to $A_2$-measures that is, for $p=2$: one may define $\SopO$ 
 which is bounded: $L^2(bD, \Ot)\to L^2(bD, \Ot)$ for any $\Ot\in A_2(bD)$ with the minimal operator norm: $\|\SopO\| =1$; on the other hand there appears to be no well-posed notion of ``$\EuScript S_{\Omega_p}$'' if  $p\neq 2$. 
The $L^p(bD, \Omega_p)$-regularity problem for $\SopO$, 
while meaningful, is, at present, unanswered for $p\neq 2$.

\subsection{Statement of the main results}

We let $\sigma$ denote induced Lebesgue measure for $bD$ and we henceforth refer to the family
\begin{equation*}
\{\Lambda\sigma\}_\Lambda \equiv \left\{\omega:= \Lambda\, \sigma,\ \Lambda\in C(bD), \  0<c(D, \Lambda)\leq\Lambda(w)\leq C(D, \Lambda)<\infty\quad \text{for any}\ w\in bD\right\}
\end{equation*}
as the {\em Leray Levi-like measures}. This is because the Leray Levi measure $\lambda$, which plays a distinguished role in the analysis
 \cite{LS2017} of the Cauchy--Leray integrals 
$\{\Cine\}_\epsilon$ and of their truncations $\{\Cine^s\}_\epsilon$, is a member of this family on account of the identity
 \begin{equation}\label{E:Leray Levi to sigma}
 d\lambda(w) =\Lambda(w)d\sigma(w),\quad w\in bD,
 \end{equation}
where $\Lambda \in C(\overline{D})$ satisfies the required bounds
$ 0< \epsilon(D) \leq \Lambda(w)\leq C(D)<\infty$ for any $ w\in bD$
as a consequence of the strong pseudoconvexity and $C^2$-regularity and boundedness of $D$. 
Hence we may equivalently express any Leray Levi-like measure $\omega$ as
\begin{equation}\label{E:LL-weights}
\omega = \varphi \lambda
\end{equation}
for some $\varphi \in C(\overline{bD})$ such that
$ 0< m(D) \leq \varphi(w)\leq M(D)<\infty$ for any $ w\in bD$.
We refer to  
Section \ref{S:2} for the precise definitions 
  and to
\cite[Lemma VII.3.9]{Ra} for the proof of \eqref{E:Leray Levi to sigma} and a discussion of the geometric significance of $\lambda$.

For any Leray Levi-like measure $\omega$,  the holomorphic Hardy space $H^2(bD, \omega)$ is defined
exactly as in the classical setting of $H^2(bD, \sigma)$, simply by replacing $\sigma$ with $\omega$,
 see e.g.,  \cite[(1.1) and (1.2)]{LS2016}. In particular $H^2(bD, \omega)$ is a closed subspace of $L^2(bD, \omega)$ and we 
 let  $\Sopo$ denote the (unique) {\em orthogonal} projection of $L^2(bD, \omega)$ onto 
  $H^2(bD, \omega)$.

Our goal is to understand the behavior of $\Sopo$ on $L^p(bD, \Op)$ where $\{\Op\}$ is any {\em $A_p$-measure} that is, 
\begin{equation}\label{E:A2weights-sigma}
 \Op =\phi_p\, \sigma\, 
 \end{equation}
where the density $\phi_p$ is a Muckenhoupt 
 $A_p$-weight. 
  {The precise definition is given in Section \ref{S:2}}; 
 here we just note  that the Leray Levi-like measures
 are strictly contained in
 the class $\{\Op\}_p$ 
 in the sense that each 
 $\omega\in \{\Lambda\sigma\}_\Lambda$
  is an $A_p$-measure for {\em every} $1<p<\infty$.  
  We may now state our main result.
\begin{theorem}\label{T:Szego-for-A_p}
Let $D\subset \mathbb C^n$, $n\geq 2$, be  a bounded, strongly pseudoconvex domain of class $C^2$; let 
$\omega$ be any Leray Levi-like measure for $bD$ and let $\Sopo$ be the Cauchy--Szeg\H o projection 
associated to $\omega$. We have that 
\begin{equation}\label{E:extrapol-Sz}
 \|\Sopo \|_{L^2(bD, \Ot)\to L^2(bD, \Ot)}\lesssim [\Ot]_{A_2}^3\quad \text{for any}\ \ \Ot \in A_2(bD),
\end{equation}
where the implied constant depends solely on $D$ and $\omega$.
\end{theorem}
As before, the power $3$ can be sharpened as $2+\delta$ for any $\delta>0$. However, it can not be reduced to $2$ due to the structure of minimal smoothness of the domain. The $L^p$-estimate \eqref{E:main-1} follows from \eqref{E:extrapol-Sz} by standard techniques, see {\cite[Theorem 9.5.3]{Gra}}.

\vskip0.1in

As mentioned earlier, extending the notion of Cauchy--Szeg\H o projection to the realm of
 $A_p$-measures requires, first of all, 
 a meaningful notion of  holomorphic Hardy space for such measures, but  this does not immediately arise from the classical theory
  \cite{STE2}; 
  here we adopt the approach of
\cite[(1.1)]{LS2016} and give the following

\begin{definition}\label{D:Hardy space}
Suppose $1\leq\, p<\infty$ and let $\Op$ be an $A_p$-measure. We define $H^p(bD, \Op)$ to be the space of functions $F$ that are holomorphic in $D$ {with}
$\mathcal N (F)\in L^p(bD, \Op)$, and set
\begin{equation}\label{E:NT-norm}
\|F\|_{H^p(bD, \Op)} := \|\mathcal N (F)\|_{L^p(bD, \Op)}.
\end{equation}
 Here $\mathcal N(F)$ denotes the non-tangential maximal function of $F$, that is
$$ \mathcal N(F)(\xi) := \sup_{z\in \Gamma_\alpha(\xi)}|F(z)|, \quad \xi\in bD,$$
{where}
 $ \Gamma_\alpha(\xi) =\{z\in D:\ |(z-\xi)\cdot \bar{\nu}_\xi| < (1+\alpha) \delta_\xi(z), |z-\xi|^2<\alpha \delta_\xi(z)\}$, with
$\bar{\nu}_\xi$ = the (complex conjugate of) the outer unit normal vector to $\xi\in bD$, and $\delta_\xi(z)=$ the minimum between the {$($Euclidean$)$} distance of $z$ to $bD$ and the distance of $z$ to the tangent 
space  at $\xi$.
\end{definition}
For the class of domains under consideration it is known that if $\Op$ is Leray Levi-like, such definition  agrees with the classical formulation \cite[(1.2)]{LS2016}; see also \cite{STE2}.
\vskip0.1in
\begin{proposition}\label{P:Hardy-def-NT}
Let $D\subset \mathbb C^n$, $n\geq 2$, be  a bounded, strongly pseudoconvex domain of class $C^2$.
Then, for any $1\leq\, p<\infty$ and any $A_p$-measure $\Op$ we have that $H^p(bD, \Op)$ is a closed subspace of $L^p(bD, \Op)$. More precisely, suppose that $\{F_n\}_n$ is a sequence of holomorphic functions in $D$ such that 
$
\|\mathcal N(F_n) - f\|_{L^p(bD, \Op)}\to 0\quad \text{as}\quad n\to\infty.
$
Then, there is an $F$  holomorphic in $D$ such that
$
\mathcal N(F)(w) = f(w)\quad
 \Op-a.e.\ w\in\ bD.
$
\end{proposition}
The proof relies on the following observation, which is of independent interest: for any $1\leq\,p<\infty$ and any $A_p$-measure $\Op$ with density $\pp$, there is $p_0 = p_0(\Op)\in (1, p)$ such that
\begin{equation}\label{E:Hp-incl}
H^p(bD, \Op)\subset H^{p_0}(bD, \sigma) \ \ \text{with}\ \ \|F\|_{H^{p_0}(bD, \sigma)}\leq C_{\Op, D}
\|F\|_{H^p(bD, \Op)},
\end{equation}
where
\begin{equation*}
C_{\Op, D} = \bigg(\,\,\int\limits_{bD}\!\pp(w)^{-\frac{p_0}{p-p_0}}\,d\sigma(w)\bigg)^{\!\!\!\frac{p-p_0}{pp_0}}.
\end{equation*}
\vskip0.1in
On the other hand in the context of $A_p$-measures the notion of Hilbert space
 is meaningful only in relation to $A_2$-measures (that is for $p=2$), hence Cauchy--Szeg\H o projections can be associated only to such measures:
on account of Proposition \ref{P:Hardy-def-NT}, for any $A_2$-measure $\Ot$ there exists a unique, orthogonal projection: $$\SopO: L^2(bD, \Ot)\to H^2(bD, \Ot)$$ 
which is naturally bounded on $L^2(bD, \Ot)$ with minimal norm $\|\SopO\| = 1$.

\vskip0.1in

\subsection{Further results} 
The proof of Theorem \ref{T:Szego-for-A_p} also requires
quantitative results for the Cauchy Leray integrals $\{\EuScript C_\epsilon\}_\epsilon$ that extend the scope of the earlier works \cite{LS2017} and \cite{DLLWW} from Leray Levi-like measures, to $A_p$-measures: these are stated in Theorem
\ref{T:5.1}
 and Proposition \ref{prop cancellation}.

\subsection{Organization of this paper.} In Section \ref{S:2} we recall the necessary background and give the proof of Proposition \ref{P:Hardy-def-NT}.  All the quantitative results pertaining to the Cauchy--Leray integral are collected in 
 in Section \ref{S:3}. 
 Theorem \ref{T:Szego-for-A_p}
  is proved in 
 Section \ref{S:4}.

\vskip0.1in

\section{Preliminaries and proof of Proposition \ref{P:Hardy-def-NT}}\label{S:2}
\setcounter{equation}{0}

In this section we introduce notations and recall certain results from  \cite{DLLWW, LS2017} that will be used throughout this paper. We will henceforth assume that $D\subset\mathbb C^n$ is 
a bounded, strongly pseudoconvex domain of class $C^2$; that is, there is
$\rho\in C^2(\mathbb C^n, \mathbb R)$ which is strictly plurisubharmonic and such that
$D=\{z\in\mathbb C^n: \rho(z)<0\}$ and 
$bD=\{w\in\mathbb C^n: \rho(w)=0\}$ with $\nabla \rho(w)\not=0$ for all $w\in bD$.
 (We refer to such $\rho$ as a {\em defining function for $D$}; see e.g., \cite{Ra} for the basic properties of defining functions. Here we assume that one such $\rho$ has been fixed once and for all.) We will throughout make use of the following abbreviated notations:
 $$
\|T\|_p\ \equiv\ \|T\|_{L^p(bD, d\mu)\to L^p(bD, d\mu)},\quad\mbox{and}\quad  \|T\|_{p, q} \ \equiv\ \|T\|_{L^p(bD, d\mu)\to L^q(bD, d\mu)} 
 $$
 where the operator $T$ and the measure $\mu$ will be clear from context.
\vskip0.1in

\noindent $\bullet$ {\em The Levi polynomial and its variants.}\quad Define
$$ \mathcal L_0(w, z) := \langle \partial\rho(w),w-z\rangle -{1\over2} \sum_{j,k} {\partial^2\rho(w) \over \partial w_j \partial w_k} (w_j-z_j)(w_k-z_k),  $$
where $\partial \rho(w)=({\partial\rho\over\partial w_1}(w),\ldots, {\partial\rho\over\partial w_n}(w))$
and we have used the notation $\langle\eta,\zeta\rangle=\sum_{j=1}^n\eta_j\zeta_j$ for $\eta=(\eta_1,\ldots, \eta_n), \zeta=(\zeta_1,\dots,\zeta_n)\in\mathbb C^n$.
The strict plurisubharmonicity of $\rho $ implies that
$$  2\operatorname{ Re} \mathcal L_0(w, z) \geq -\rho(z)+c |w-z|^2,  $$
for some $c>0$, whenever $w\in bD$ and $z\in \bar D$ is sufficiently close to $w$.
We next define
\begin{align}\label{g0}
    g_0(w,z) := \chi \mathcal L_0+ (1-\chi) |w-z|^2
\end{align}
where $\chi=\chi(w,z)$ is a $C^\infty$-smooth cutoff function with $\chi=1$ when $|w-z|\leq \mu/2$ and $\chi=0$ if $|w-z|\geq \mu$.
Then for $\mu$ chosen sufficiently small (and then kept fixed throughout), we have that
\begin{equation}\label{E:c}
 \operatorname{ Re}g_0(w,z)\geq c(-\rho(z)+ |w-z|^2)
 \end{equation}
for $z$ in $\bar D$ and $w$ in $bD$, with $c$  a positive constant; we will refer to $g_0(w, z)$ as {\em the modified Levi polynomial}.
Note that $g_0(w, z)$ is polynomial in the variable $z$,  whereas in the variable $w$ it has no smoothness beyond mere continuity. To amend for this lack of regularity, 
 for each $\epsilon>0$ one considers a variant $g_\epsilon$ defined as follows.  Let $\{\tau_{jk}^\epsilon(w)\}$ be an $n\times n$-matrix  of $C^1$ functions such that
$$\sup_{w\in bD}\Big|{\partial^2\rho(w)\over\partial w_j\partial w_k}- \tau_{jk}^\epsilon(w)\Big|\leq\epsilon,\quad 1\leq j,k\leq n.$$
Set
\begin{align}\label{cepsilon}
c_\epsilon:=\sup_{w\in bD, 1\leq j,k\leq n} \big|\nabla \tau_{jk}^\epsilon(w)\big|.
\end{align}
For the convenience of our statement and proof, we may choose those $\{\tau_{jk}^\epsilon(w)\}$ such that 
\begin{align}\label{cepsilon bound}
c_\epsilon\lesssim \epsilon^{-1}.
\end{align}
where the implicit constant is independent of $\epsilon$.
We also set
$$ \mathcal L_\epsilon(w, z) = \langle \partial\rho(w),w-z\rangle -{1\over2} \sum_{j,k}\tau_{jk}^\epsilon(w) (w_j-z_j)(w_k-z_k),  $$
and define
$$    g_\epsilon(w,z) = \chi \mathcal L_\epsilon+ (1-\chi) |w-z|^2, \quad z,w\in\mathbb C^n.  $$
Now $g_\epsilon$ is of class $C^1$ in the variable $w$, and
$$\left|g_0(w,z)-g_\epsilon(w,z) \right|\lesssim \epsilon |w-z|^2,\quad w\in bD, z\in \overline{D}.$$
We assume that $\epsilon$ is sufficiently small (relative to the constant $c$ in \eqref{E:c}), and this gives that
\begin{equation}\label{E:epsilon-0}
\left|g_0(w,z) \right|\leq  \left|g_\epsilon(w,z)\right|\leq \tilde C\left|g_0(w,z) \right|,\quad w, z\in bD
\end{equation}
where the constants $C$ and $\tilde C$ are independent of $\epsilon$; see \cite[Section 2.1]{LS2017}.

\vskip0.1in

\noindent $\bullet$ {\em The Leray--Levi measure for $bD$.}\quad 
Let $j^*$ denote the pullback under the inclusion $$j:bD\hookrightarrow\mathbb C^n.$$ Then  the linear functional
\begin{align}\label{lambda}  
f\mapsto {1\over (2\pi i)^n} \int\limits_{bD} f(w) j^*(\partial \rho \wedge (\bar\partial \partial \rho)^{n-1})(w)=: \int\limits_{bD} f(w)d\lambda(w) 
\end{align}
where $f\in C(bD)$, defines  a measure $\lambda$ with positive density given by
$$   d\lambda(w) = {1\over (2\pi i)^n} j^*(\partial\rho \wedge (\bar\partial \partial\rho)^{n-1}) (w).  $$
We point out that the definition of $\lambda$ depends upon the choice of defining function for $D$, which here has been fixed once and for all; hence we refer to $\lambda$ as ``the'' {\em Leray--Levi measure}.
\vskip0.1in

\noindent $\bullet$ {\em A space of homogeneous type.}\quad Consider the function 
\begin{equation}\label{E:quasi-dist}
{\tt d}(w,z) := |g_0(w,z)|^{1\over2},\quad w, z\in bD.
\end{equation} 
It is known \cite[(2.14)]{LS2017} that
$$
|w-z|\lesssim {\tt d}(w, z)\lesssim |w-z|^{1/2},\quad w, z\in bD
$$
and from this it follows that the space of H\"older-type functions \cite[(3.5)]{LS2017}:
\begin{equation}\label{E:Holder}
|f(w) - f(z)|\lesssim {\tt d}(w, z)^\alpha\quad \mbox{for some}\ 0<\alpha\leq 1\ \ \mbox{and for all}\ w, z\in bD
\end{equation}
is dense in $L^p(bD, \omega)$, $1<p<\infty$ for any Leray Levi-like measure see \cite[Theorem 7]{LS2017}.

\vskip0.1in

It follows from \eqref{E:epsilon-0} that
\begin{align}\label{gd}
{\color{black} \tilde C{\tt d}(w,z)^{2}\leq |g_\epsilon(w,z)|\leq C {\tt d}(w,z)^{2}, \quad w, z\in bD}
 \end{align}
for any $\epsilon$ sufficiently small. It is shown in \cite[Proposition 3]{LS2017} that  ${\tt d}(w, z)$ is a quasi-distance: there exist constants $A_0>0$ and  $C_{\tt d}>1$ such that for all $w,z,z'\in bD$,
\begin{align}\label{metric d}
\left\{
                \begin{array}{ll}
                  1)\ \ {\tt d}(w,z)=0\quad {\rm iff}\quad w=z;\\[5pt]
                  2)\ \ A_0^{-1} {\tt d}(z,w)\leq  {\tt d}(w,z) \leq A_0 {\tt d}(z,w);\\[5pt]
                  3)\ \ {\tt d}(w,z)\leq C_{\tt d}\big( {\tt d}(w,z') +{\tt d}(z',z)\big).
                \end{array}
              \right.
\end{align}

\smallskip

Letting  $ B_r(w) $ denote the  boundary balls  determined via the quasi-distance ${\tt d }$,
\begin{align}\label{ball} 
B_r(w) :=\{ z\in bD:\ {\tt d}(z,w)<r \}, \quad {\rm where\ } w\in bD,
\end{align} 
we have that
\begin{align}\label{E:omegab}
c_\omega^{-1} r^{2n}\leq \omega\big(B_r(w) \big)\leq c_\omega r^{2n},\quad 0<r\leq 1,
\end{align}
for some $c_\omega>1$, see \cite[p. 139]{LS2017}.
It follows that the triples $\{bD, {\tt d}, \omega\}$, for any Leray Levi-like measure $\omega$,  are spaces of homogeneous type, where the measures $\omega$ have the doubling property:
{
\begin{lemma}\label{measure lambda}
The Leray Levi-like measures $\omega$ on $bD$ are doubling, i.e., there is a positive constant $C_\omega$ such that for all $x\in bD$ and $0<r\leq1$,
$$ 0<\omega(B_{2r}(w))\leq C_\omega\omega(B_{r}(w))<\infty. $$
Furthermore, there exist constants
$\epsilon_\omega\in(0,1)$ and $C_\omega>0$ such that 
$$ \omega( B_r(w)\backslash B_r(z) ) +  \omega( B_r(z)\backslash B_r(w) ) \leq C_\omega \left( { {\tt d}(w,z) \over r}  \right)^{\!\!\epsilon_\omega}   $$
for all $w,z\in bD$ such that  ${\tt d}(w,z)\leq r\leq1$.
\end{lemma}
\begin{proof}
The proof is an immediate consequence of \eqref{E:omegab}.
\end{proof}
}

\vskip0.1in

\noindent $\bullet$\quad {\em A family of Cauchy-like integrals.}\quad 
 In \cite[Sections 3 and 4]{LS2017} an ad-hoc family $\{\mathbf C_\epsilon\}_\epsilon$ of Cauchy-Fantappi\`e integrals is introduced (each determined by the aforementioned denominators $g_\epsilon(w, z)$) whose corresponding boundary operators $\{\EuScript C_\epsilon\}_\epsilon$ play a crucial role in the analysis
 of $L^p(bD, \lambda)$-regularity of the Cauchy--Szeg\H o projection. 
 We henceforth refer to 
 $\{\Cine\}_\epsilon$ as the {\em Cauchy-Leray integrals}; we record here a few relevant points for later reference.
\vskip0.1in
\begin{itemize}
 \item[{\tt [i.]}] Each $\EuScript C_\epsilon$ 
 admits a primary decomposition in terms of an ``essential part'' $\EuScript C_\epsilon^\sharp$ and a ``remainder'' $\EuScript R_\epsilon$, which are used in the proof of the $L^2(bD, \omega)$-regularity of $\EuScript C_\epsilon$. However, at this stage the magnitude of the parameter $\epsilon$ plays no role
 (this is because of the ``uniform'' estimates \eqref{gd}) and we temporarily drop reference to $\epsilon$ and simply write  $\EuScript C$ in lieu of $\EuScript C_\epsilon$; $C(w, z)$ for $C_\epsilon(w, z)$, etc.. Thus
\begin{align}\label{C operator}
   \EuScript C =  \EuScript C^\sharp+ \EuScript R,  
\end{align}
with a corresponding decomposition for the integration kernels:
\begin{equation}\label{E: C kernel}
C(w,z)=C^\sharp(w,z)+ R(w,z).
\end{equation}
\vskip0.07in

\item[]
The ``essential'' kernel $C^\sharp(w,z)$  satisfies standard size and smoothness conditions that ensure
the boundedness of $\EuScript C^\sharp$ in $L^2(bD, \omega)$ by a $T(1)$-theorem for the space of homogenous type $\{bD, {\tt d}, \omega\}$. On the other hand, the ``remainder'' kernel $R(w,z)$ satisfies improved size and smoothness conditions granting
 that the corresponding operator $\EuScript R$ is bounded in $L^2(bD, \omega)$ by elementary considerations; see \cite[Section 4]{LS2017}.
\vskip0.1in

\item[{\tt [ii.]}] 
One then turns to the Cauchy--Szeg\H o projection, for which $L^2(bD, \omega)$-regularity is trivial but $L^p(bD, \omega)$-regularity, for $p\neq 2$, is not. 
It is in this stage that the size of $\epsilon$ in the definition of the Cauchy-type boundary operators of item  {\tt [i.]} is relevant.
It turns out that each $\EuScript C_\epsilon$ admits a further, ``finer'' decomposition into (another) ``essential'' part and (another) ``reminder'', which are obtained by truncating the integration kernel $C_\epsilon (w, z)$ by a smooth cutoff function $\chi_\epsilon^s(w, z)$ that equals 1 when ${\tt d}(w, z) <s = s(\epsilon)$. One has:
\begin{equation}\label{E:Cs-op}
\EuScript C_\epsilon = \EuScript C^s_\epsilon + \EuScript R^s_\epsilon
\end{equation}
where
\begin{equation}\label{E:Ess-small}
 \|  (\EuScript C^s_\epsilon)^\dagger -  \EuScript C^s_\epsilon\|_{p} \lesssim \epsilon^{1/2} M_p
 \end{equation}
for any $1<p<\infty$, where $\displaystyle{M_p = {p\over p-1} +p}$. Here and henceforth, the upper-script ``$\dagger$'' denotes adjoint in $L^2(bD, \omega)$ (hence $(\EuScript C^s_\epsilon)^\dagger$ is the adjoint of $\EuScript C^s_\epsilon$ in $L^2(bD, \omega)$); see \cite[Proposition 18]{LS2017}.
Furthermore $\EuScript R^s_\epsilon$ and $(\EuScript R^s_\epsilon)^\dagger$ are controlled by ${\tt d}(w, z)^{-2n+1}$ and therefore are easily seen to be bounded
\begin{equation}\label{E:RsBdd}
\EuScript R^s_\epsilon,\ \ (\EuScript R^s_\epsilon)^\dagger:
L^1(bD, \omega)\to L^\infty(bD, \omega),
\end{equation}
 see \cite[(5.2) and comments thereafter]{LS2017}. 

\end{itemize}

\vskip0.1in

\vskip0.1in

\noindent $\bullet$ {\em Muckenhoupt weights on $bD$.} \quad Let {$p\in(1, \infty)$. A non-negative locally integrable function $\psi$ is called an
\emph{$A_p(bD, \sigma)$-weight}, if
\begin{align*}
[\psi]_{A_p(bD, \sigma)}:=\sup_{B} \langle \psi\rangle_B\langle \psi^{1-p'}\rangle_B^{p-1}<\infty,
\end{align*}
where the supremum is taken over all balls $B$ in $bD$, and 
$ \displaystyle{\langle \phi\rangle_B:={1\over \sigma(B)}\int\limits_{B}\phi (z)\,d\sigma(z)}$. Moreover,  $\psi$ is called an
\emph{$A_1(bD, \sigma)$-weight} if $[\psi]_{A_1(bD, \sigma)}:=\inf\{C\geq0: \langle \psi\rangle_B \leq C \psi(x), \forall x\in B, \forall {\rm\ balls\ } B\in bD \} <\infty$.

\vskip0.1in
Similarly, one can define the \emph{$A_p(bD, \lambda)$-weight} for $1\leq p<\infty$.

As before, the identity \eqref{E:Leray Levi to sigma} grants that $$A_p(bD, \sigma) = A_p(bD, \lambda)\quad \text{with}\quad
[\psi]_{A_p(bD, \sigma)}\approx [\psi]_{A_p(bD, \lambda)},$$ thus we will henceforth simply write
$A_p(bD)$ and $[\psi]_{A_p(bD)}$. At times it will be more convenient to work with $A_p(bD, \lambda)$, and in this case we will refer to its members as
{\em $A_p$-like weights}.
\vskip0.1in

\begin{proof}[\textit{\textbf{Proof of Proposition \ref{P:Hardy-def-NT}}}]
To streamline notations, we write $\Omega$ for $\Omega_p$, and $\psi$ for $\psi_p$. It is clear that $H^p(bD,\Omega)$ is a subspace of $L^p(bD,\Omega)$, where the density function $\psi$ of $\Omega$ is in $A_p$.

Our first claim is that 
 for every $F\in H^p(bD,\Omega)$, 
{the non-tangential (also known as admissible) limit}
 $F^b(w)$ exists $\Omega$-a.e. $w\in bD$.  In fact, note that for $\psi\in A_p$, there exists $1<p_1<p$ such that $\psi$ is in $A_{p_1}$. Set $p_0=p/p_1$. Then it is clear that $1<p_0<p$. Let $P=p/p_0$ and $P'$ be the conjugate of $P$, which is $P'=p/(p-p_0)$.
So we get $\psi^{ 1-P'} = \psi^{-p_0/(p-p_0)} \in A_{P'}=A_{p/(p-p_0)}$.

Hence, for $F\in H^p(bD,\Omega)$,
\begin{align*}
\|F\|_{H^{p_0}(bD,\lambda)}&=\bigg(\int_{bD} |\mathcal N(F)(w)|^{p_0}\psi(w)^{p_0\over p} \psi(w)^{-{p_0\over p}} d\lambda(w)\bigg)^{1\over p_0}\\
&\leq \bigg(\int_{bD} |\mathcal N(F)(w)|^p\psi(w) d\lambda(w)\bigg)^{1\over p}\bigg(\int_{bD} \psi(w)^{-{p_0\over p-p_0}} d\lambda(w)\bigg)^{p-p_0\over pp_0}\\
&= \|F\|_{H^2(bD,\Omega)} \Big(\psi^{-{p_0\over 2-p_0}} (bD)\Big)^{p-p_0\over pp_0}.
\end{align*}
Since $\psi^{-p_0/(p-p_0)} \in A_{p/(p-p_0)}$ and $b D$ is compact we have that $\psi^{-{p_0\over p-p_0}} (bD)$ is finite.

Hence, we see that $ H^p(bD,\Omega)\subset  H^{p_0}(bD,\lambda)$. 
Thus
$F$ has admissible limit $F^b$ for $\lambda$--a.e. $z\in bD$ (\cite[Theorem 10]{STE2}), 
and hence $F$ has admissible limit $F^b$ for $\Omega$--a.e. $z\in bD$ since the measure $\Omega$ (with the density function $\psi$: $d\Omega(z)=\psi(z)d\lambda(z)$) is absolutely continuous. 
So the boundary function $F^b$ exists.

Next, and from the definition of the non-tangential maximal function $\mathcal N(F)$, we have that 
$$  |F^b(z)|\leq \mathcal N(F)(z)\quad {\rm for\ {\Omega}-a.e. \ } z\in bD. $$
Thus, $F^b\in L^p(bD,\Omega)$.
Also note that
{with the same methods of}
 \cite{STE2} {one can show that} 
$$   \mathcal N(F)(z)\lesssim  M(F^b)(z)\quad {\rm for\ {\Omega}-a.e. \ } z\in bD, $$
where $M(F^b)$ is the Hardy--Littlewood maximal function on the boundary $bD$.
Since the maximal function is bounded on $L^{p}(bD,\Omega)$, we obtain that 
$$ \|F\|_{ H^p(bD,\Omega) } = \| \mathcal N(F) \|_{L^p(bD,\Omega)} \lesssim \| M(F^b) \|_{L^p(bD,\Omega)}\leq 
C[\psi]_{A_p} \|F^b\|_{L^p(bD,\Omega)}.$$

Suppose  now that $\{F_n\}$ is a sequence in $H^p(bD,\Omega)$ and $f\in L^p(bD,\Omega)$
such that $$\|\mathcal N(F_n)-f\|_{L^p(bD,\Omega)}\to0. $$
Then, it is clear that 
$$\|\mathcal N(F_n)-f\|_{L^{p_0}(bD,\lambda)}\leq \|\mathcal N(F_n)-f\|_{L^{p}(bD,\Omega)} \Big(\psi^{-{p_0\over p-p_0}} (bD)\Big)^{p-p_0\over pp_0}\to0. $$
Since $H^{p_0}(D,\lambda)$ is a proper subspace of $L^{p_0}(bD,\lambda)$, then in particular $\{F_n\}_n$ is  in $H^{p_0}(D,\lambda)$ and $f$ is  in $L^{p_0}(bD,\lambda)$, 
we see that 
there is $F$ that is holomorphic in $D$ and such that 
$$
F^b(w) = f(w)\quad \lambda-a.e.\ w\in\ bD.
$$
Again, this implies that
$$
F^b(w) = f(w)\quad \Omega-a.e.\ w\in\ bD.
$$
  Moreover, invoking the weighted boundedness of the Hardy--Littlewood maximal function, we see that
$$
\|\mathcal N(F)\|_{L^{p}(bD, \Omega)}\lesssim \|f\|_{L^p(bD, \Omega)}.
$$

The proof of Proposition \ref{P:Hardy-def-NT} is complete.
\end{proof}

\section{Quantitative estimates for the Cauchy--Leray integral
}\label{S:3}
\setcounter{equation}{0}

As before, in the proofs of all statements in this section we adopt the shorthand $\Omega$ for $\Omega_p$, and $\psi$ for $\psi_p$.
\begin{theorem}\label{T:5.1} Let $D\subset \mathbb C^n$, $n\geq 2$, be  a bounded, strongly pseudoconvex domain of class $C^2$. Then the Cauchy-type integral $\EuScript C_\epsilon$ is bounded on $L^p(bD, \Op)$ for any $0<\epsilon<\epsilon(D)$, any $1<p<\infty$ and any $A_p$-measure $\Op$, 
  {with}
 \begin{align}\label{C quantitative bound}
 \|\EuScript C_\epsilon\|_{L^p(bD,\Op)\to L^p(bD,\Op)}\  \lesssim
 \, c_\epsilon\cdot [\Op]_{A_p}^{\max\{1,{1\over p-1}\}},
  \end{align}
where the implied constant depends on $p$ and $D$,  but is independent of  $\epsilon$ or $\Op$, and $c_\epsilon$ is the constant in \eqref{cepsilon}.
\end{theorem}
It follows that for any $A_2$-measure $\Ot$, the $L^2(bD, \Ot)$-adjoint
$\EuScript C_\epsilon^\spadesuit$ 
is also bounded on $L^p(bD, \Op)$ with same bound.
\begin{proof}

To begin with, we first recall that 
the Cauchy integral operator $\EuScript C_\epsilon$ can be split into the essential part and remainder, that is,
$\EuScript C_\epsilon=\EuScript C_\epsilon^\sharp+\EuScript R_\epsilon.$
Denote by  $ C^\sharp_\epsilon(w,z)$ and $R_\epsilon(w,z)$ the kernels of $\EuScript C_\epsilon^\sharp$ and $\EuScript R_\epsilon$, respectively.

Recall from \cite[(4.9), (4.18) and (4.19)]{LS2017} that $ C^\sharp_\epsilon(w,z)$ is a standard Calder\'on--Zygmund kernel, i.e. there exists a positive constant $\mathcal A_1$ such that for every $w,z\in bD$ with $w\not=z$,
\begin{align}\label{gwz}
\left\{
                \begin{array}{ll}
                  a)\ \ |C^\sharp_\epsilon(w,z)|\leq \mathcal A_1 {\displaystyle1\over\displaystyle{\tt d}(w,z)^{2n}};\\[5pt]
                  b)\ \ |C^\sharp_\epsilon(w,z) - C^\sharp_\epsilon(w',z)|\leq c_\epsilon\cdot \mathcal A_1 {\displaystyle {\tt d}(w,w')\over \displaystyle {\tt d}(w,z)^{2n+1} },\quad {\rm if}\ {\tt d}(w,z)\geq c{\tt d}(w,w');\\[5pt]
                  c)\ \ |C^\sharp_\epsilon(w,z) - C^\sharp_\epsilon(w,z')|\leq \mathcal A_1  {\displaystyle {\tt d}(z,z')\over \displaystyle {\tt d}(w,z)^{2n+1} },\quad {\rm if}\ {\tt d}(w,z)\geq c {\tt d}(z,z')
                \end{array}
              \right.
\end{align}
for an appropriate constant $c>0$, where ${\tt d}(z,w)$ is a quasi-distance suitably adapted to $bD$, and $c_\epsilon$ is the constant in \eqref{cepsilon}. And hence, the $L^p(bD)$ boundedness  ($1<p<\infty$) of $\EuScript C_\epsilon^\sharp$ follows from a version of the $T(1)$ Theorem.  Moreover, from \cite[(4.9)]{LS2017} we also get that there exists a positive constant $\mathcal A_2$ such that for every $w,z\in bD$ with $w\not=z$,
\begin{align}\label{gwz1}
                   |C^\sharp_\epsilon(w,z)|\geq \mathcal A_2 {\displaystyle1\over\displaystyle{\tt d}(w,z)^{2n}}.
\end{align}
However, the kernel $R_\epsilon(w,z)$ of $\EuScript R_\epsilon$ satisfies a size condition and a smoothness condition for only one of the variables as follows: there exists
a positive constant $C_R$ (independent of $\epsilon$) such that for every $w,z\in bD$ with $w\not=z$,
\begin{align}\label{cr}
\left\{
                \begin{array}{ll}
                  d)\ \ |R_\epsilon(w,z)|\leq C_R {\displaystyle1\over \displaystyle{\tt d}(w,z)^{2n-1}};\\[5pt]
                  e)\ \ |R_\epsilon(w,z)-R_\epsilon(w,z')|\leq C_R {\displaystyle{\tt d}(z,z')\over\displaystyle {\tt d}(w,z)^{2n}},\quad {\rm if\ }  {\tt d}(w,z)\geq c_R {\tt d}(z,z').
                \end{array}
              \right.
\end{align}

Since the kernel of $\EuScript C_\epsilon^\sharp$ is a standard Calder\'on--Zygmund kernel on $bD\times bD$, according to \cite{Hy0} (see also \cite{Lacey0}),  we can obtain that $\EuScript C_\epsilon^\sharp$ is bounded on  $L^p(bD,\Omega)$ with 
\begin{align}\label{Csharp quantitative bound}
\|\EuScript C_\epsilon^\sharp\|_{L^p(bD,\Omega)\to L^p(bD,\Omega)}
\lesssim  c_\epsilon\cdot \mathcal A_1  [\psi]_{A_p}^{\max\{1,{1\over p-1}\}},
\end{align}
where  $c_\epsilon$ and $\mathcal A_1$ are as in \eqref{gwz}.
Thus, it suffices to show that $\EuScript R_\epsilon$ is bounded on  $L^p(bD,\Omega)$ with the appropriate quantitative estimate.

To see this, we claim that for every $f\in L^p(bD,\Omega),\tilde z\in bD$, there exists a $q\in (1,p)$ such that
\begin{align}\label{Tsharp}
\Big|\left( {\EuScript R_\epsilon} (f)\right)^\#(\tilde z)\Big| \lesssim \Big(M(|f|^q)(\tilde z)\Big)^{1/q},
\end{align}
where $F^\#$ is the sharp maximal function of $F$ as recalled in Section 2, and the implied constant depends on $p$, $D$ and $c_R$.

We now show this claim. Since $bD$ is bounded, there exists $\overline C>0$ such that for any $B_r(z)\subset bD$ we have $r<\overline C$.
For any $\tilde z \in bD$, let us fix a ball $B_r=B_r(z_0)\subset bD$ containing $\tilde z$,  and let $z$ be any point of $B_r$.
Now take $j_0=\lfloor\log_2 {\overline C\over r}\rfloor+1$. Since ${\tt d}$ is a quasi-distance,
 there exists $i_0\in\mathbb Z^+$, independent of $z$, $r$, such that ${\tt d}(w,z)>c_R r$ whenever $w\in bD\setminus B_{2^{i_0} r}$, where $c_R$ is in \eqref{cr}.
We then write
\begin{align*}
{\EuScript R_\epsilon}(f)(z)
&=\big(\EuScript R_\epsilon\big(f\chi_{ bD \cap B_{2^{i_0} r}} \big)(z)
-\EuScript R_\epsilon\big(f\chi_{bD\setminus B_{2^{i_0} r}} \big)(z)=: I(z)+I\!I(z).
\end{align*}

For the term $I$, by using H\"older's inequality and the fact that $\EuScript R_\epsilon$ is bounded on $L^q(bD,\lambda)$, $1<q<\infty$, we have
 \begin{align*}
{1\over \lambda(B_r)}\int\limits_{B_r}|I(z)-I_{B_r}|d\lambda(z)
&\leq2\left( {1\over \lambda(B_r)}\int\limits_{B_r}\big|\EuScript R_\epsilon\big(f\chi_{bD\cap B_{2^{i_0} r}} \big)(z)\big|^qd\lambda(z) \right)^{1\over q}\\
&\lesssim \left( {1\over \lambda(B_r)}\int\limits_{bD\cap B_{2^{i_0} r}} |f(z)|^qd\lambda(z) \right)^{1\over q}\\
&\lesssim  \big(M(| f|^q)(\tilde z) \big)^{1\over q}.
\end{align*}

To estimate $I\!I$, observe that if $i_0\geq j_0$, then $bD\setminus B_{2^{i_0} r}=\emptyset$ and $| I\!I(z)-I\!I(z_0)|=0$.
If $i_0<j_0$, then we have
\begin{align*}
\left| I\!I(z)-I\!I(z_0)\right|
&=\big|\EuScript R_\epsilon\big(f\chi_{bD\setminus B_{2^{i_0} r}} \big)(z)
-\EuScript R_\epsilon\big(f\chi_{bD\setminus B_{2^{i_0} r}} \big)(z_0)\big|\\
&\leq \int\limits_{bD\setminus B_{2^{i_0} r}}\big| R_\epsilon(w,z)-R_\epsilon(w,z_0) \big|  |f(w)|d\lambda(w)\\
&\leq {\tt d}(z,z_0) \int\limits_{bD\setminus B_{2^{i_0} r}}{1\over {\tt d}(w, z_0)^{2n}} |f(w)|d\lambda(w)\\
&\leq  r \left( \int\limits_{bD\setminus B_{2^{i_0} r}}{1\over {\tt d}(w, z_0)^{2n}}d\lambda(w)\right)^{1\over q'} \left(\int\limits_{bD\setminus B_{2^{i_0} r}}{1\over {\tt d}(w, z_0)^{2n}}|f(w)|^qd\lambda(w) \right)^{1\over q}.
\end{align*}
 Since $bD$ is bounded,  we can obtain
\begin{align*}
\int\limits_{bD\setminus B_{2^{i_0} r}}{1\over {\tt d}(w, z_0)^{2n}}|f(w)|^qd\lambda(w)
&\leq\sum_{j=i_0}^{j_0}\,\,\int\limits_{2^j r\leq d(w, z_0)\leq 2^{j+1} r}{1\over {\tt d}(w, z_0)^{2n}}|f(w)|^qd\lambda(w) \\
&\leq \sum_{j=i_0}^{j_0}{1\over (2^j r)^{2n}}
\int\limits_{d(w, z_0)\leq 2^{j+1} r}|f(w)|^qd\lambda(w)\\
&\lesssim\sum_{j=i_0}^{j_0}  {1\over \lambda(B_{2^{j+1} r})}\int\limits_{B_{2^{j+1} r}}|f(w)|^qd\lambda(w)\\
&\lesssim  j_0 M(|f|^q)(\tilde z).
\end{align*}
Similarly, we have
\begin{align*}
 \int\limits_{bD\setminus B_{2^{i_0} r}}{1\over {\tt d}(w, z_0)^{2n}}d\lambda(w)
 &\lesssim\sum_{j=i_0}^{j_0}  {1\over \lambda(B_{2^{j+1} r})}\int\limits_{B_{2^{j+1} r}} d\lambda(w)\lesssim j_0.
\end{align*}
Thus, we get that 
$\left|I\!I(z)-I\!I(z_0)\right|\lesssim  r j_0 \left(M(|f|^q)(\tilde z)\right)^{1\over q}.
$
Therefore,
\begin{align*}
{1\over \lambda(B_r)}\int\limits_{B_r}\big| I\!I(z)-I\!I_{B_r}\big|d\lambda(z)
&\leq {2\over \lambda(B_r)}\int\limits_{B_r}\big| I\!I(z)-I\!I(z_0)\big|d\lambda(z)\lesssim r  j_0 \left(M(|f|^q)(\tilde z)\right)^{1\over q}\\
&\lesssim r  \left(\log_2\left( {\overline C\over r} \right)+1\right) \left(M(|f|^q)(\tilde z)\right)^{1\over q}\lesssim \left(M(|f|^q)(\tilde z)\right)^{1\over q},
\end{align*}
where the last inequality comes from the fact that $ r \log_2\left( {\overline C\over r}\right)$ is uniformly bounded.  Combining the estimates on $I$ and $I\!I$, we see that the claim \eqref{Tsharp} holds. 

We now prove that $\EuScript R_\epsilon$ is bounded on $L^p(bD,\Omega).$
In fact, for $f\in L^p(bD,\Omega)$
\begin{align}\label{R Lp}
\left\|{\EuScript R_\epsilon}(f)\right\|_{L^p(bD,\Omega)}^{p}&\leq C\Big(\Omega(bD)({\EuScript R_\epsilon}(f)_{bD})^p+\|{\EuScript R_\epsilon}(f)^\#\|_{L^p(bD,\Omega)}^{p}\Big)\\
&\leq C\Big(\Omega(bD)({\EuScript R_\epsilon}(f)_{bD})^p+\|\left(M(|f|^q)\right)^{1\over q}\|_{L^p(bD,\Omega)}^{p}\Big) \nonumber\\
&\lesssim \Omega(bD)({\EuScript R_\epsilon}(f)_{bD})^p+ [\psi]_{A_p}^{p\cdot \max\{1,{1\over p-1}\}}\|f\|_{L^p(bD,\Omega)}^{p},\nonumber
\end{align}
where the second inequality follows from \eqref{Tsharp}  and the  last inequality
follows from the fact that the Hardy--Littlewood function is bounded on $L^p(bD,\Omega)$. 
We point out that 
\begin{align}\label{R L1}
\Omega(bD)({\EuScript R_\epsilon}(f)_{bD})^p \lesssim \|f\|_{L^p(bD,\Omega)}^{p}.
\end{align}
In fact, by H\"older's inequality and that $\EuScript R_\epsilon$ is bounded on $L^q(bD,\lambda)$, $1<q<\infty$, we have
\begin{align*}
\Omega(bD)({\EuScript R_\epsilon}(f)_{bD})^p
&\leq  \Omega(bD) \left({1\over \lambda(bD)}\int\limits_{bD} \big|\EuScript R_\epsilon(f)(z)\big|^q  d\lambda(z)\right)^{p\over q}\\
&\lesssim \Omega(bD) \left({1\over \lambda(bD)}\int\limits_{bD} \big|f(z)\big|^q  d\lambda(z)\right)^{p\over q}\\
&\lesssim \Omega(bD) \inf_{z\in bD} \Big(M(|f|^q)(z)\Big)^{p\over q}\\
&\lesssim  \int\limits_{bD}\Big(M(|f|^q)(z)\Big)^{p\over q}d\Omega(z)\\
&\lesssim [\Omega]_{A_p}^{p\cdot \max\{1,{1\over p-1}\}}\|f\|_{L^p(bD,\Omega)}^{p}.
\end{align*}
Therefore, \eqref{R L1} holds, which, together with \eqref{R Lp}, implies that
$\EuScript R_\epsilon$ is bounded on $L^p(bD,\Omega)$ with the correct quantitative bounds. 
This, together with \eqref{Csharp quantitative bound}, gives that \eqref{C quantitative bound} holds. 
The proof of Theorem \ref{T:5.1} is complete.
\end{proof}

We now turn to the proof of the new cancellation \eqref{E:newcancellation-2}.

\begin{proposition}\label{prop cancellation}
For any fixed $0<\epsilon<\epsilon (D)$ as in \cite{LS2017}, there exists $s=s(\epsilon)>0$ such that  
   \begin{equation}\label{E:s-to-dag-s}
  \|(\EuScript C_\epsilon^s)^\dagger-\EuScript C_\epsilon^s\|_{L^p(bD, \Op)\to L^p(bD, \Op)}\
  {\lesssim}\ \epsilon^{1/2} [\Op]_{A_p}^{\max\{1,{1\over p-1}\}}
  \end{equation}
 \noindent  $\text{for any}\ 1<p<\infty \  \text{and for any} \ A_p \text{-measure},  \Op$ where the implied constant depends on $D$ and $p$ but is independent of $\Op$ and of $\epsilon$.
  As before, here  $(\EuScript C_\epsilon^s)^\dagger$ 
   denotes the adjoint
  in $L^2(bD, \omega)$.
  \end{proposition} 
 
 \begin{remark}\label{remark power} 
 We point out that the term $\epsilon^{1/2}$ can be improved to $\epsilon^{\delta}$ for any fixed small $\delta>0$, according to \cite[Remark D]{LS2017} via choosing $\beta$ there arbitrarily close to 1.
  \end{remark}

\begin{proof}{Proof of Proposition \ref{prop cancellation}}

Recall from \cite[(5.7)]{LS2017}, that $\EuScript C_\epsilon^s$ is given by 
$$\EuScript C_\epsilon^s(f)(z) = \EuScript C_\epsilon \big( f(\cdot) \chi_s(\cdot,z) \big)(z), \quad z\in bD,$$
where $\chi_s(w,z)$ is the  cutoff function given by
$ \chi_s(w,z)=\widetilde \chi_{s,w}(z) \widetilde \chi_{s,z}(w)  $
with 
$$ \widetilde \chi_{s,w}(z) = \chi\Bigg( { {\rm Im} \langle \partial\rho(w) , w-z\rangle \over cs^2} + i { |w-z|^2\over cs^2 } \Bigg). $$
Here $\chi$ is a non-negative $C^1$-smooth function on $\mathbb C$ such that $\chi(u+iv)=1$ if $|u+iv|\leq 1/2$,
$\chi(u+iv)=0$ if $|u+iv|\geq 1$, and furthermore $|\nabla\chi(u+iv)|\lesssim 1$.

Then we also have the essential part and the remainder of $\EuScript C_\epsilon^s$, which are given by 
$$\EuScript C_\epsilon^{\sharp,s} (f)(z) = \EuScript C_\epsilon^{\sharp}  \big( f(\cdot) \chi_s(\cdot,z) \big)(z), \quad z\in bD,$$
and hence we have $ \EuScript C_\epsilon^s=\EuScript C_\epsilon^{\sharp,s}+ \EuScript R_\epsilon^{\sharp,s} $, where $\EuScript R_\epsilon^{\sharp,s}$ is the remainder. Further,  
$ \big(\EuScript C_\epsilon^s\big)^*=\big(\EuScript C_\epsilon^{\sharp,s}\big)^*+ \big(\EuScript R_\epsilon^{\sharp,s}\big)^* $.  
To prove \eqref{E:s-to-dag-s}, we 
note that $(\EuScript C_\epsilon^s)^\dagger = \varphi^{-1} (\EuScript C_\epsilon^s)^* \varphi$, where 
$\varphi$ is the density function of $\omega$ satisfying \eqref{E:LL-weights}. Thus,
$ (\EuScript C_\epsilon^s)^\dagger-\EuScript C_\epsilon^s =(\EuScript C_\epsilon^s)^*-\EuScript C_\epsilon^s
- \varphi^{-1} [\varphi, (\EuScript C_\epsilon^s)^*] $,
which gives
\begin{align*} 
& \|(\EuScript C_\epsilon^s)^\dagger-\EuScript C_\epsilon^s\|_{L^p(bD, \Omega_p)\to L^p(bD, \Omega_p)}\\
&\leq 
\|(\EuScript C_\epsilon^s)^*-\EuScript C_\epsilon^s\|_{L^p(bD, \Omega_p)\to L^p(bD, \Omega_p)} + \| \varphi^{-1} [\varphi, (\EuScript C_\epsilon^s)^*] \|_{L^p(bD, \Omega_p)\to L^p(bD, \Omega_p)}.
\end{align*}

Thus, we  claim that
   \begin{equation}\label{E:s-to-star-s}
  \|(\EuScript C_\epsilon^s)^*-\EuScript C_\epsilon^s\|_{L^p(bD, \Omega_p)\to L^p(bD, \Omega_p)}
  \leq \epsilon^{1/2} [\psi]_{A_p}^{\max\{1,{1\over p-1}\}}M (p,D, \omega)
  \end{equation}
  and that
   \begin{equation}\label{E: C phi}
  \| \varphi^{-1} [\varphi, (\EuScript C_\epsilon^s)^*] \|_{L^p(bD, \Omega_p)\to L^p(bD, \Omega_p)}
  \leq \epsilon^{1/2} [\psi]_{A_p}^{\max\{1,{1\over p-1}\}}M (p,D, \omega).
  \end{equation}
 
We first prove \eqref{E:s-to-star-s}.
To begin with, we write $ \EuScript C_\epsilon^s-\big(\EuScript C_\epsilon^s\big)^* =   \EuScript A_\epsilon^s+ \EuScript B_\epsilon^s$, where 
$$ \EuScript A_\epsilon^s=\EuScript C_\epsilon^{\sharp,s}- \big(\EuScript C_\epsilon^{\sharp,s}\big)^*\qquad{\rm and}\qquad  \EuScript B_\epsilon^s=\EuScript R_\epsilon^{\sharp,s}- \big(\EuScript R_\epsilon^{\sharp,s}\big)^*.$$

Let $A_\epsilon^s(w,z)$ be the kernel of $ \EuScript A_\epsilon^s$. Then, from \cite[(5.10) and (5.9)]{LS2017}, we see that
\begin{align*}
\left\{
                \begin{array}{ll}
                  1)\ \ |A_\epsilon^s(w,z)|\lesssim \epsilon {\displaystyle1\over\displaystyle{\tt d}(w,z)^{2n}}\qquad {\rm for\ any\ } s\leq s(\epsilon);\\[5pt]
                  2)\ \ |A_\epsilon^s(w,z) - A_\epsilon^s(w',z)|\lesssim \epsilon^{1/2} {\displaystyle {\tt d}(w,w')^{1/2}\over \displaystyle {\tt d}(w,z)^{2n+1/2} }\qquad {\rm for\ any\ } s\leq s(\epsilon);\\[5pt]
                \end{array}
              \right.
\end{align*}
for  ${\tt d}(w,z)\geq c{\tt d}(w,w')$ with an appropriate constant $c>0$. Moreover, 2) is true for $w$ and $z$ interchanged, that is, 
$$3)\  \ |A_\epsilon^s(w,z) - A_\epsilon^s(w,z')|\lesssim \epsilon^{1/2} {\displaystyle {\tt d}(z,z')^{1/2}\over \displaystyle {\tt d}(w,z)^{2n+1/2} },\quad {\rm for\ any\ } s\leq s(\epsilon).$$

From the kernel estimates, we see that $\epsilon^{-1/2} \EuScript A_\epsilon^s$, $s\leq s(\epsilon)$, is a standard Calder\'on--Zygmund operator on $bD\times bD$, according to \cite{Hy0} (see also \cite{Lacey0}),  we can obtain that $\epsilon^{-1/2} \EuScript A_\epsilon^s$ is bounded on  $L^p(bD,\Omega_p)$ with 
$$\|\epsilon^{-1/2} \EuScript A_\epsilon^s\|_{L^p(bD, \Omega_p)\to L^p(bD, \Omega_p)}
\leq [\psi]_{A_p}^{\max\{1,{1\over p-1}\}} M (p,D, \omega),$$
where the constant $M (p,D, \omega)$ depends only on $p$, $D$ and $\omega$.
Hence, we obtain that 
\begin{align}\label{A Lp}
\| \EuScript A_\epsilon^s\|_{L^p(bD, \Omega_p)\to L^p(bD, \Omega_p)}
\leq \epsilon^{1/2} [\psi]_{A_p}^{\max\{1,{1\over p-1}\}}M (p,D, \omega) \qquad {\rm for\ any\ } s\leq s(\epsilon).
\end{align}

Thus, it suffices to consider $\EuScript B_\epsilon^s$. Note that now the kernel $R_\epsilon^s(w,z)$ of $ \EuScript R_\epsilon^{\sharp,s}$ 
is given by
$$ R_\epsilon^s(w,z) = R_\epsilon(w,z) \chi_s(w,z), $$
where $R_\epsilon(w,z)$ is the kernel of $\EuScript R_\epsilon$ satisfying \eqref{cr}.
Thus, it is easy to see that $R_\epsilon^s(w,z)$ satisfies the following size estimate 
\begin{align*}
                  4)\ \ |R_\epsilon^s(w,z)|\leq \widetilde c_\epsilon {\displaystyle\chi_s(w,z)\over \displaystyle{\tt d}(w,z)^{2n-1}},
\end{align*}
where the constant $\widetilde c_\epsilon$ is large, depending on $\epsilon$.
For the regularity estimate of $R_\epsilon^s(w,z)$ on $z$, we get that for ${\tt d}(w,z)\geq c_{R} {\tt d}(z,z')$,
\begin{align*}
&|R_\epsilon^s(w,z)-R_\epsilon^s(w,z')|\\
&\leq | R_\epsilon(w,z)-R_\epsilon(w,z') |\cdot \chi_s(w,z) + |R_\epsilon(w,z')|\cdot |\chi_s(w,z)-\chi_s(w,z')|\\
&\leq \widetilde c_\epsilon {\displaystyle{\tt d}(z,z')\over \displaystyle{\tt d}(w,z)^{2n}}\cdot \chi_s(w,z) + \widetilde c_\epsilon {\displaystyle 1\over \displaystyle{\tt d}(w,z)^{2n-1}}\cdot |\chi_s(w,z)-\chi_s(w,z')|,
\end{align*}
where the last inequality follows from \eqref{cr}. Note that from the definition of $\chi_s(w,z)$, we get that
$|\chi_s(w,z)-\chi_s(w,z')|$ vanishes unless ${\tt d}(w,z)\approx {\tt d}(w,z')\approx s$. Moreover, under this condition, we further have 
$$|\chi_s(w,z)-\chi_s(w,z')|\lesssim {\displaystyle{\tt d}(z,z')\over \displaystyle{\tt d}(w,z)}.$$
Combining these estimates, we obtain that 
\begin{align*}
                  5)\ \ |R_\epsilon^s(w,z)-R_\epsilon^s(w,z')|\leq \widetilde c_\epsilon {\displaystyle{\tt d}(z,z')\ \chi_s(w,z)\over\displaystyle {\tt d}(w,z)^{2n}},\quad {\rm if\ }  {\tt d}(w,z)\geq c_{R} {\tt d}(z,z').
\end{align*}

We now consider the operator $\EuScript  T= s^{-1}\EuScript R_\epsilon^{\sharp,s}$. Then we claim that
\begin{align}\label{R Lp R/s claim}
\left\|{\EuScript T}(f)\right\|_{L^p(bD,\Omega_p)}\leq  \widetilde c_\epsilon [\psi]_{A_p}^{\max\{1,{1\over p-1}\}}M (p,D, \omega) \|f\|_{L^p(bD,\Omega_p)} .
\end{align}
To prove \eqref{R Lp R/s claim}, following the proof of Theorem \ref{T:5.1}, 
we first show that for every $f\in L^p(bD,\Omega_p),\tilde z\in bD$, there exists a $q\in (1,p)$ such that
\begin{align}\label{Tsharp R/s}
\Big|\left( \EuScript  T (f)\right)^\#(\tilde z)\Big| \lesssim \widetilde c_\epsilon\Big(M(|f|^q)(\tilde z)\Big)^{1/q},
\end{align}
where $F^\#$ is the sharp maximal function of $F$ as recalled in Section 2.

We now show \eqref{Tsharp R/s}. Since $bD$ is bounded, there exists $\overline C>0$ such that for any $B_r(z)\subset bD$ we have $r<\overline C$.
For any $\tilde z \in bD$, let us fix a ball $B_r=B_r(z_0)\subset bD$ containing $\tilde z$,  and let $z$ be any point of $B_r$.
 Since ${\tt d}$ is a quasi-distance,
 there exists $i_0\in\mathbb Z^+$, independent of $z$, $r$, such that ${\tt d}(w,z)>c_R r$ whenever $w\in bD\setminus B_{2^{i_0} r}$, where $c_R$ is in \eqref{cr}.
We then write
\begin{align*}
 \EuScript  T(f)(z)
&=  \EuScript  T\big(f\chi_{ bD \cap B_{2^{i_0} r}} \big)(z)
-\big(f\chi_{bD\setminus B_{2^{i_0} r}} \big)(z)=: I(z)+I\!I(z).
\end{align*}

For the term $I$, by using H\"older's inequality we have
 \begin{align*}
{1\over \lambda(B_r)}\int\limits_{B_r}|I(z)-I_{B_r}|d\lambda(z)
&\leq2\left( {1\over \lambda(B_r)}\int\limits_{B_r}\big| \EuScript  T\big(f\chi_{bD\cap B_{2^{i_0} r}} \big)(z)\big|^qd\lambda(z) \right)^{1\over q}\\
&\lesssim \widetilde c_\epsilon\left( {1\over \lambda(B_r)}\int\limits_{bD\cap B_{2^{i_0} r}} |f(z)|^qd\lambda(z) \right)^{1\over q}\\
&\lesssim \widetilde c_\epsilon \big(M(| f|^q)(\tilde z) \big)^{1\over q},
\end{align*}
where the second inequality follows from the boundedness in \cite[(4.22)]{LS2017} since the kernel of $\EuScript  T$
satisfies the conditions in \cite[(4.22)]{LS2017}.

To estimate $I\!I$,  by using condition 5), we have
\begin{align*}
\left| I\!I(z)-I\!I(z_0)\right|
&=\big|\EuScript T\big(f\chi_{bD \setminus B_{2^{i_0} r}} \big)(z)
-\EuScript T\big(f\chi_{bD\setminus B_{2^{i_0} r}} \big)(z_0)\big|\\
&\leq \int\limits_{bD\setminus B_{2^{i_0} r}} {\widetilde c_\epsilon\over s} {\displaystyle{\tt d}(z,z_0)\ \chi_s(w,z_0)\over\displaystyle {\tt d}(w,z_0)^{2n}} |f(w)|d\lambda(w)\\
&\leq \int\limits_{(bD\setminus B_{2^{i_0} r} ) \cap B_s} {\widetilde c_\epsilon\over s} {\displaystyle{\tt d}(z,z_0)\ \over\displaystyle {\tt d}(w,z_0)^{2n}} |f(w)|d\lambda(w),
\end{align*}
where $B_s:=B_s(z_0)$ and the last inequality follows from the property of the function $\chi_s(w,z_0)$. Thus, we see that if $2^{i_0}r\geq s$, then the last term is zero. So we just need to consider $2^{i_0}r<s$. In this case, we have
\begin{align*}
\left| I\!I(z)-I\!I(z_0)\right|
&\leq {\widetilde c_\epsilon\over s}\ {\tt d}(z,z_0) \int\limits_{B_s\setminus B_{2^{i_0} r}}{1\over {\tt d}(w, z_0)^{2n}} |f(w)|d\lambda(w)\\
&\leq  {\widetilde c_\epsilon\ r\over s} \left( \int\limits_{B_s\setminus B_{2^{i_0} r}}{1\over {\tt d}(w, z_0)^{2n}}d\lambda(w)\right)^{1\over q'} \left(\int\limits_{B_s\setminus B_{2^{i_0} r}}{1\over {\tt d}(w, z_0)^{2n}}|f(w)|^qd\lambda(w) \right)^{1\over q}.
\end{align*}
 We take $j_0=\lfloor\log_2 {\displaystyle s\over\displaystyle  r}\rfloor+1$.  Continuing the estimate:
\begin{align*}
\int\limits_{B_s\setminus B_{2^{i_0} r}}{1\over {\tt d}(w, z_0)^{2n}}|f(w)|^qd\lambda(w)
&\leq\sum_{j=i_0}^{j_0}\ \int\limits_{2^j r\leq d(w, z_0)\leq 2^{j+1} r}{1\over {\tt d}(w, z_0)^{2n}}|f(w)|^qd\lambda(w) \\
&\leq \sum_{j=i_0}^{j_0}{1\over (2^j r)^{2n}}
\int\limits_{d(w, z_0)\leq 2^{j+1} r}|f(w)|^qd\lambda(w)\\
&\lesssim\sum_{j=i_0}^{j_0}  {1\over \lambda(B_{2^{j+1} r})}\int\limits_{B_{2^{j+1} r}}|f(w)|^qd\lambda(w)\\
&\lesssim  j_0 M(|f|^q)(\tilde z).
\end{align*}
Similarly, we have
\begin{align*}
 \int\limits_{B_s\setminus B_{2^{i_0} r}}{1\over {\tt d}(w, z_0)^{2n}}d\lambda(w)
 &\lesssim\sum_{j=i_0}^{j_0}  {1\over \lambda(B_{2^{j+1} r})}\int\limits_{B_{2^{j+1} r}} d\lambda(w)\lesssim j_0.
\end{align*}
Thus, we get that 
$\left|I\!I(z)-I\!I(z_0)\right|\lesssim  \widetilde c_\epsilon {r\over s}  j_0 \left(M(|f|^q)(\tilde z)\right)^{1\over q}.
$
Therefore,
\begin{align*}
{1\over \lambda(B_r)}\int\limits_{B_r}\big| I\!I(z)-I\!I_{B_r}\big|d\lambda(z)
&\leq {2\over \lambda(B_r)}\int\limits_{B_r}\big| I\!I(z)-I\!I(z_0)\big|d\lambda(z)\lesssim \widetilde c_\epsilon {r\over s}    j_0 \left(M(|f|^q)(\tilde z)\right)^{1\over q}\\
&\lesssim \widetilde c_\epsilon {r\over s}    \left(\log_2\left( { s\over r} \right)+1\right) \left(M(|f|^q)(\tilde z)\right)^{1\over q}\lesssim  \widetilde c_\epsilon\left(M(|f|^q)(\tilde z)\right)^{1\over q},
\end{align*}
where the last inequality comes from the fact that $ A \left(\log_2( {1\over A})+1\right)$ is uniformly bounded for $A\in (0,1]$.

Combining the estimates on $I$ and $I\!I$, we see that the claim \eqref{Tsharp R/s} holds. We now prove that $\EuScript T$ is bounded on $L^p(bD,\Omega_p).$
In fact, for $f\in L^p(bD,\Omega_p)$
\begin{align}\label{R Lp R/s}
\left\|{\EuScript T}(f)\right\|_{L^p(bD,\Omega_p)}^{p}&\lesssim \Big(\Omega_p(bD)({\EuScript T}(f)_{bD})^p+\|{\EuScript T}(f)^\#\|_{L^p(bD,\Omega_p)}^{p}\Big)\\
&\lesssim \Big(\Omega_p(bD)({\EuScript T}(f)_{bD})^p+\widetilde c_\epsilon^p\|\left(M(|f|^q)\right)^{1\over q}\|_{L^p(bD,\Omega_p)}^{p}\Big) \nonumber\\
&\leq \Omega_p(bD)({\EuScript T}(f)_{bD})^p+ \widetilde c_\epsilon^p [\psi]_{A_p}^{p\cdot \max\{1,{1\over p-1}\}}M (p,D, \omega) \|f\|_{L^p(bD,\Omega_p)}^{p},\nonumber
\end{align}
where the second inequality follows from \eqref{Tsharp R/s}  and the  last inequality
follows from the fact that the Hardy--Littlewood function is bounded on $L^p(bD,\Omega)$ with the constant $M (p,D, \omega)$ depending only on $p$, $D$ and $\omega$. 
We point out that 
\begin{align}\label{R L1 R/s}
\Omega_p(bD)({\EuScript T}(f)_{bD})^p \leq \widetilde c_\epsilon [\Op]_{A_p}^{p\cdot \max\{1,{1\over p-1}\}} M (p,D, \omega)\|f\|_{L^p(bD,\Omega_p)}^{p}.
\end{align}

In fact, by H\"older's inequality and the fact that $\EuScript T$ is bounded on $L^q(bD,\lambda)$ for all $1<q<\infty$ (from \cite[(4.22)]{LS2017}), we have
\begin{align*}
\Omega_p(bD)({\EuScript T}(f)_{bD})^p
&\leq  \Omega_p(bD) \left({1\over \lambda(bD)}\int\limits_{bD} \big|\EuScript T(f)(z)\big|^q  d\lambda(z)\right)^{p\over q}\\
&\lesssim \widetilde c_\epsilon^p\Omega_p(bD) \left({1\over \lambda(bD)}\int\limits_{bD} \big|f(z)\big|^q  d\lambda(z)\right)^{p\over q}\\
&\lesssim \widetilde c_\epsilon^p \Omega_p(bD) \inf_{z\in bD} \Big(M(|f|^q)(z)\Big)^{p\over q}\\
&\lesssim  \widetilde c_\epsilon^p\int\limits_{bD}\Big(M(|f|^q)(z)\Big)^{p\over q}d\Omega_p(z)\\
&\leq \widetilde c_\epsilon^p [\Op]_{A_p}^{p\cdot \max\{1,{1\over p-1}\}} M (p,D, \omega)\|f\|_{L^p(bD,\Omega_p)}^{p}.
\end{align*}
Therefore, \eqref{R L1 R/s} holds, which, together with \eqref{R Lp R/s}, implies that
the claim \eqref{R Lp R/s claim} holds.

As a consequence, we obtain that for every $1<p<\infty$,
\begin{align}\label{R Lp R/s claim outcome}
\left\|{ \EuScript R_\epsilon^{\sharp,s}}(f)\right\|_{L^p(bD,\Omega_p)}\leq s\cdot \widetilde c_\epsilon \cdot[\Op]_{A_p}^{\max\{1,{1\over p-1}\}}M (p,D, \omega)\|f\|_{L^p(bD,\Omega_p)}.
\end{align}
By using the fact that $\Omega_p^{-{1\over p-1}}$ is in $A_{p'}$ when $\Omega_p$ is in $A_p$ ($p'$ is the conjugate index of $p$), and by using duality, we obtain that 
\begin{align}\label{R Lp R/s claim outcome2}
\left\|\big({ \EuScript R_\epsilon^{\sharp,s}}\big)^*(f)\right\|_{L^p(bD,\Omega_p)}\leq s\cdot \widetilde c_\epsilon \cdot [\Op]_{A_p}^{\max\{1,{1\over p-1}\}}M (p,D, \omega)\|f\|_{L^p(bD,\Omega_p)}.
\end{align}
Hence, we have 
\begin{align}\label{B Lp}
\left\|\EuScript B_\epsilon^s(f)\right\|_{L^p(bD,\Omega_p)}\leq s\cdot \widetilde c_\epsilon \cdot [\Op]_{A_p}^{\max\{1,{1\over p-1}\}}M (p,D, \omega)\|f\|_{L^p(bD,\Omega_p)}.
\end{align}
Since $\epsilon $ is a fixed small positive constant, we see that $$s\cdot \widetilde c_\epsilon \approx \epsilon^{1/2}$$
if $s$ is sufficiently small. Thus, we obtain that
\begin{align}\label{B Lp small}
\left\|\EuScript B_\epsilon^s(f)\right\|_{L^p(bD,\Omega_p)}\lesssim \epsilon^{1/2} \cdot [\Op]_{A_p}^{\max\{1,{1\over p-1}\}}M (p,D, \omega)\|f\|_{L^p(bD,\Omega_p)}.
\end{align}
Combining the estimates of \eqref{A Lp} and \eqref{B Lp small}, we obtain that the claim
\eqref{E:s-to-star-s} holds.  

We now prove \eqref{E: C phi}. To begin with, following \cite[Section 6.2]{LS2017}, we consider a partition of
$\mathbb C^n$ into disjoint cubes of side-length $\gamma$, given by 
$\mathbb C^n =\cup_{k \in \mathbb Z^{2n}} Q^\gamma_k$.
Then we revert to our domain $D$. For a fixed $\gamma >0$, we write $1_k$ for the characteristic function of 
$Q^\gamma_k\cap bD$. 
Then we consider $1_k(z) (\EuScript C_\epsilon^s)^* (1_j f)(z)$. 
we note that if $z$ is not in the support of $f$, then
   \begin{equation}
   1_k(z) (\EuScript C_\epsilon^s)^* (1_j f)(z) = 1_k(z)\int\int_{bD} \overline{ C_\epsilon (z,w)} \chi_s(w,z) f(w) 1_j (w) d\lambda(w).
  \end{equation}
Thus, by choosing $s$ small enough and $\gamma = cs$, where $c$ is a fixed positive constant, we see that
$1_k(z) (\EuScript C_\epsilon^s)^* (1_j f)(z)$ vanishes if $Q_j^\gamma$ and $Q_k^\gamma$ do not touch. 
Next, following the proof in \cite[Section 6.3]{LS2017}, and combining our result Theorem \ref{T:5.1}  we see that 
\begin{align*}\|1_k(\EuScript C_\epsilon^s)^* (1_j f) \|_{L^p(bD,\Omega_p)}&\lesssim \epsilon \|\EuScript C_\epsilon^s\|_{L^p(bD,\Omega_p) \to L^p(bD,\Omega_p)}  \|f \|_{L^p(bD,\Omega_p)} \\
&\leq \epsilon [\Op]_{A_p}^{\max\{1,{1\over p-1}\}} M (p,D, \omega)\|f \|_{L^p(bD,\Omega_p)}.
\end{align*}
As a consequence, by using \cite[Lemma 24]{LS2017}, we obtain that our claim
 \eqref{E: C phi} holds.

The proof of Proposition \ref{prop cancellation} is complete.
\end{proof}

\vskip0.2in

\section{
 Proofs of Theorem \ref{T:Szego-for-A_p}}\label{S:4}
\setcounter{equation}{0}

 We start with \eqref{E:LS-2} and thus for any $\epsilon >0$ we write
  \begin{equation*}
  \Sopo g = \EuScript C_\epsilon\Tinv g\ +\ 
   \Sopo\circ\big((\EuScript R^s_\epsilon)^\dagger-\EuScript R^s_\epsilon\big)\circ 
   \Tinv g\ =: \mathfrak A_\epsilon g +\mathfrak B_\epsilon g,
  \end{equation*}
  where
  $$\mathcal T^s_\epsilon h := \bigg(I - \big((\EuScript C_\epsilon^s)^\dagger - \EuScript C_\epsilon^s\big)\bigg)h.$$

We fix $\Ot$ arbitrarily. To streamline the notations, we write $M(D,\omega)$ for  {$M (2,D, \omega)$ as in \eqref{E:s-to-dag-s}}, see 
Proposition \ref{prop cancellation}.
We now choose $\epsilon = \epsilon (\Ot)$ such that
\begin{equation}\label{E:epsilon}
 \frac14\leq \epsilon^{1/2} [\Ot]_{A_2}M (D, \omega) \leq \frac12,
\end{equation}
where $\pt$ is the density of $\Ot$. Thus, Proposition \ref{prop cancellation} grants
$$
\|\big((\EuScript C_\epsilon^s)^\dagger - \EuScript C_\epsilon^s\big)^j g\|_{L^2(bD, \Ot)} \leq \frac{1}{2^j}\|g\|_{L^2(bD, \Ot)},\quad j=0, 1, 2,\ldots.
$$
Hence for $\epsilon = \epsilon (\Ot)$ as in \eqref{E:epsilon} we have that
\begin{equation}\label{E:bd}
\|\Tinv g\|_{L^2(bD, \Ot)}\leq 2 \|g\|_{L^2(bD, \Ot)},
\end{equation}
 and by Theorem \ref{T:5.1} (for $p=2$) we conclude that
$$
\| \mathfrak A_\epsilon g\|_{L^2(bD, \Ot)} \leq C_1(D, \omega)\cdot   {\color{red}c_\epsilon}\cdot[\Ot]_{A_2}\,\|g\|_{L^2(bD, \Ot)}.
$$
Combining the estimate of $c_\epsilon$ in \eqref{cepsilon bound} and the choice of $\epsilon$ in \eqref{E:epsilon},
we obtain that 
\begin{align}\label{Aepsilon quantitative bound}
\| \mathfrak A_\epsilon g\|_{L^2(bD, \Ot)} \lesssim  [\Ot]_{A_2}^3\,\|g\|_{L^2(bD, \Ot)}.
\end{align}

We point out that from Remark \ref{remark power}, the term $\epsilon^{1/2}$ can be improved to $\epsilon^{\delta}$ for any small $\delta>0$. Hence, the term $[\Ot]_{A_2}^3$ in \eqref{Aepsilon quantitative bound} could be improved to $[\Ot]_{A_2}^{2+\delta}$ for any small $\delta>0$.

We next proceed to bound the norm of $ \mathfrak B_\epsilon g$; to this end we recall that the reverse H\"older inequality for $\Ot =\pt d\lambda$ (\cite[Theorem 9.2.2]{Gra})
\begin{equation}\label{E:reverse-holder}
\left(\frac{1}{\lambda(bD)} \ \int\limits_{bD}\pt^{1+\gamma}\!(z)\,d\lambda(z)\right)^{\!\!\frac{1}{\gamma +1}}\,\leq \,
C(\Ot)\,\frac{1}{\lambda(bD) } 
\int\limits_{bD}\pt (z) d\lambda (z)
\end{equation}
is true for for some $\gamma = \gamma(\Ot)>0$.

Recall that by keeping track of the constants (see \cite[(9.2.6) and (9.2.7)]{Gra}),
we have that
$$
\sup\limits_{\Ot\in A_2} \gamma(\Ot)\leq1
$$
and that there exists an absolute positive constant $C_2$ such that 
$$
\sup\limits_{\Ot\in A_2} C(\Ot)\leq C_2^2<\infty.
$$
Hence, we obtain that
\begin{equation}\label{E:reverse-holder 1}
\left(\ \int\limits_{bD}\pt^{1+\gamma}\!(z)\,d\lambda(z)\right)^{\!\!\frac{1}{\gamma +1}}\,\leq \,
C_2^2\,\frac{1}{\lambda(bD)^{\frac{\gamma}{\gamma +1}}}
\int\limits_{bD}\pt (z) d\lambda (z).
\end{equation}

Recall that, for Leray Levi-like measures we have  $d\omega (z) = \varphi(z)\, d\lambda (z),$
where
$$
0<m_\omega\leq \varphi (z)\leq M_\omega <\infty\quad \text{for any}\ \  z\in bD.
$$
 Hence, using the shorthand
 $$
 B_\epsilon g = \Sopo H,\quad H:= \big((\EuScript R^s_\epsilon)^\dagger-\EuScript R^s_\epsilon\big)h,\quad
 h:= \Tinv g
 $$
 we have
 $$
 \|B_\epsilon g\|_{L^2(bD, \Ot)} \leq m_\omega^{-\frac{\gamma}{\gamma+1}}\left(\ \int\limits_{bD}|\Sopo H|^2(z)\,\varphi (z)^{\frac{\gamma}{\gamma+1}}\pt (z)\,d\lambda(z)\right)^{\!\!\frac12}.
 $$
 H\"older's inequality for $\displaystyle{q:= \frac{\gamma+1}{\gamma}}$, where $\gamma =\gamma(\Ot)$ is as in \eqref{E:reverse-holder}, now gives that
 $$
  \|B_\epsilon g\|_{L^2(bD, \Ot)} \leq m_\omega^{-\frac{\gamma}{\gamma+1}}
  \left(\ \int\limits_{bD}|\Sopo (H)|^{\frac{2(\gamma+1)}{\gamma}}(z)\, d\omega(z)\right)^{\!\!\frac{\gamma}{2(\gamma +1)}}\!\left(\ \int\limits_{bD}\pt^{1+\gamma}(z)\, d\lambda (z)\right)^{\!\!\frac{1}{2(1+\gamma)}}.
 $$
 Comparing the above with \eqref{E:reverse-holder 1} we obtain
 $$
  \|B_\epsilon g\|_{L^2(bD, \Ot)} \leq m_\omega^{-\frac{\gamma}{\gamma+1}} C_2
  \|\Sopo (H)\|_{L^p(bD, \omega)}
  \frac{1}{\lambda(bD)^{\frac{\gamma}{2(1+\gamma)}}}\left(\ \int\limits_{bD}\pt (z)\, d\lambda(z)\right)^{\!\!\frac12},
 $$
 where $\displaystyle{p:= \frac{2(1+\gamma)}{\gamma}}$. But
 $$
 \|\Sopo (H)\|_{L^p(bD, \omega)} \leq C (\omega, D) \|H\|_{L^p(bD, \omega)}
 $$
 by the $L^p(bD, \omega)$-regularity of $\Sopo$ \cite{LS2017}, since $H$ is shorthand for $\big((\EuScript R^s_\epsilon)^\dagger-\EuScript R^s_\epsilon\big)h$, and  each of $(\EuScript R^s_\epsilon)$ and $(\EuScript R^s_\epsilon)^\dagger$
 takes $L^1(bD, \omega)$ to $L^\infty(bD, \omega)$ (again by \cite{LS2017}). We obtain
 \begin{align*}
 \|\Sopo (H)\|_{L^p(bD, \omega)} &\leq C(\omega, D)\Omega_p(bD)^{\!\frac1p}\|h\|_{L^1(bD, \omega)}\\
 &\leq C (\omega, D) \Omega_p(bD)^{\!\frac1p}M_\omega \|h\|_{L^2(bD, \Ot)}\bigg(\ \int\limits_{bD}\pt^{-1}(z)d\lambda(z)\bigg)^{\!\!\frac12}.
 \end{align*}
 But $h:= \Tinv g$ and choosing $\epsilon = \epsilon (\Ot)$ as in \eqref{E:epsilon} we have that 
 $$
 \|h\|_{L^2(bD, \Ot)}\leq 2 \|g\|_{L^2(bD, \Ot)}\quad \text{by}\quad \eqref{E:bd}\, .
 $$
Combining all the pieces we obtain
\begin{equation*}\label{E:bound B}
 \|B_\epsilon g\|_{L^2(bD, \Ot)} \leq 2\,C_2\,C (\omega, D) M_\omega
 \bigg(m_\omega^{-1} \!\!
 \left(\frac{\Omega_p(bD)}{\lambda(bD)}\right)^{\!\!\frac12}
 \bigg)^{\frac{\gamma}{\gamma+1}}\, [\Ot]_{A_2}^{\frac12}
\|g\|_{L^2(bD, \Ot)}.
\end{equation*}
Hence the desired bound for $ \|B_\epsilon g\|_{L^2(bD, \Ot)}$ holds true because
$$
 \bigg(m_\omega^{-1} 
 \left(\frac{\Omega_p(bD)}{\lambda(bD)}\right)^{\!\!\frac12}
 \bigg)^{\!\!{\frac{\gamma}{\gamma+1}}}\! \leq\, C_3 (\omega, D)<\infty \quad \text{for any}\ 0\leq\gamma<\infty.
$$
The proof of Theorem \ref{T:Szego-for-A_p} is complete {once we recall that $[\Ot]_{A_2}\geq 1$}. \qed

\vskip0.4in

{\bf Acknowledgements.}
X. Duong and Ji Li are supported by the Australian Research Council (award no. DP 220100285). L. Lanzani is 
a member of the INdAM group GNAMPA.
B.\,D. Wick
is partially supported by the National Science Foundation (award no. DMS--1800057) and the Australian Research Council, award no. DP 220100285. We thank the mathematical research institute MATRIX in Australia where part of this research was performed.

\vskip0.1in
\bibliographystyle{amsplain}

\end{document}